\documentclass[11pt]{amsart}
\usepackage[T1]{fontenc}
\usepackage{amssymb,amscd,amsthm,latexsym}
\allowdisplaybreaks[3]
\usepackage{graphicx}
\usepackage[british]{babel}
\usepackage[all]{xy}
\usepackage{color}
\usepackage{mathabx}
\usepackage{calrsfs}
\usepackage{booktabs}

\theoremstyle{plain}
\newtheorem{theorem}[subsection]{Theorem}
\newtheorem{lemma}[subsection]{Lemma}
\newtheorem{proposition}[subsection]{Proposition}
\newtheorem{corollary}[subsection]{Corollary}

\theoremstyle{definition}
\newtheorem{definition}[subsection]{Definition}

\theoremstyle{remark}
\newtheorem{example}[subsection]{Example}

\newtheorem{remark}[subsection]{Remark}

\newdir{ >}{
 @{}*!/-10pt/@{>} }
\newdir{ |>}{
 @{}*!/-5.5pt/@{|}*!/-10pt/:(1,-.2)@^{>}*!/-10pt/:(1,+.2)@_{>} }
\newdir{ >>}{{}*!/8pt/@{|}*!/3.5pt/:(1,-.2)@^{>}*!/3.5pt/:(1,+.2)@_{>}}

\newcommand{\defn}[1]{\emph{#1}}

\newcommand{\caixa}[1]{\framebox{\scriptsize\ensuremath{\mathrm{#1}}}}

\newcommand{\C}{\ensuremath{\mathbb{C}}}
\newcommand{\X}{\ensuremath{\mathbb{X}}}

\newcommand{\R}{\ensuremath{\mathbb{R}}}
\newcommand{\N}{\ensuremath{\mathbb{N}}}
\newcommand{\Z}{\ensuremath{\mathbb{Z}}}
\newcommand{\Xx}{\ensuremath{\mathsf{X}}}

\newcommand{\Grp}{\ensuremath{\mathsf{Grp}}}
\newcommand{\Ab}{\ensuremath{\mathsf{Ab}}}
\newcommand{\Ord}{\ensuremath{\mathsf{Ord}}}
\newcommand{\OrdAb}{\ensuremath{\mathsf{OrdAb}}}
\newcommand{\OOrdAb}{\ensuremath{\mathbb{O}}\ensuremath{\mathsf{rdAb}}}
\newcommand{\OrdGrp}{\ensuremath{\mathsf{OrdGrp}}}
\newcommand{\OOrdGrp}{\ensuremath{\mathbb{O}}\ensuremath{\mathsf{rdGrp}}}
\newcommand{\VGrp}{V\mbox{-}\ensuremath{\mathsf{Grp}}}
\newcommand{\VAb}{V\mbox{-}\ensuremath{\mathsf{Ab}}}
\newcommand{\OVAb}{\mathbb{V}\mbox{-}\ensuremath{\mathsf{Ab}}}

\newcommand{\Set}{\ensuremath{\mathsf{Set}}}
\newcommand{\Pt}{\ensuremath{\mathsf{Pt}}}
\newcommand{\Ptr}{\ensuremath{\mathsf{Ptr}}}
\newcommand{\Ptl}{\ensuremath{\mathsf{Ptl}}}
\newcommand{\Eq}{\ensuremath{\mathsf{Eq}}}

\def\mathrmdef#1{\expandafter\def\csname#1\endcsname{{\rm#1}}}
\mathrmdef{co}\mathrmdef{op}

\newcommand{\ito}{\dashrightarrow}
\newcommand{\la}{\langle}
\newcommand{\ra}{\rangle}

\renewcommand{\prec}{\preccurlyeq}

\def\pullback{
 \ar@{-}[]+R+<6pt,-1pt>;[]+RD+<6pt,-6pt>%
 \ar@{-}[]+D+<1pt,-6pt>;[]+RD+<6pt,-6pt>}
\def\ophalfsplitpullback{%
 \ar@{-}[]+R+<6pt,-1pt>;[]+RD+<6pt,-6pt>%
 \ar@{-}[]+D+<.5ex,-6pt>;[]+RD+<6pt,-6pt>}

\usepackage{chngcntr}
\counterwithin*{equation}{section}

\begin{document}

\title{On lax protomodularity for $\Ord$-enriched categories}

\author{Maria Manuel Clementino}
\address{CMUC, Department of Mathematics, University of
Coimbra, 3001-501 Coimbra, Portugal}\thanks{}
\email{mmc@mat.uc.pt}

\author{Andrea Montoli}
\address{Dipartimento di Matematica ``Federigo Enriques'', Universit\`{a} degli Studi di Milano, Via Saldini 50, 20133 Milano, Italy}
\thanks{}
\email{andrea.montoli@unimi.it}

\author{Diana Rodelo}
\address{Departamento de Matem\'atica, Faculdade de Ci\^{e}ncias e Tecnologia, Universidade
do Algarve, Campus de Gambelas, 8005-139 Faro, Portugal and CMUC, Department of Mathematics, University of Coimbra, 3001-501 Coimbra, Portugal}
\thanks{The first and third author acknowledge partial financial assistance by the Centre for Mathematics
of the University of Coimbra -- UIDB/00324/2020, funded by the Portuguese Government through
FCT/MCTES}
\email{drodelo@ualg.pt}



\begin{abstract}
Our main focus concerns a possible lax version of the algebraic property of protomodularity for $\Ord$-enriched categories. Our motivating example is the category $\OrdAb$ of preordered abelian groups; indeed, while abelian groups form a protomodular category, $\OrdAb$ does not.

Having in mind the role of comma objects in the enriched context, we consider some of the characteristic properties of protomodularity with respect to comma objects instead of pullbacks. We show that the equivalence between protomodularity and certain properties on pullbacks also holds when replacing conveniently pullbacks by comma objects in any finitely complete category enriched in $\Ord$, and propose to call lax protomodular such $\Ord$-enriched categories. We conclude by studying this sort of lax protomodularity for $\OrdAb$, equipped with a suitable $\Ord$-enrichment, and show that $\OrdAb$ fulfills the equivalent lax protomodular properties with respect to the weaker notion of \emph{precomma object}; we call such categories lax preprotomodular.
\end{abstract}

\maketitle

\section*{Introduction}
The category $\OrdGrp$ of preordered groups was recently studied in~\cite{CM-FM}, where the authors exhibited several of their algebraic and homological properties. This category differs from the category of internal groups in the category $\Ord$ of preorders and monotone maps, since the inversion morphism of the group structure is not necessarily monotone (it is, in fact, anti-monotone). As a consequence, many of the nice algebraic properties of groups fail to hold in that context. Indeed, in~\cite{CM-FM} it was shown that $\OrdGrp$ shares several algebraic properties with the category $\mathsf{Mon}$ of monoids, but fails to share some of the algebraic properties of the category $\Grp$ of groups. Similar observations can be made for the category $\OrdAb$ of preordered abelian groups and $\Ab$ of abelian groups.

This work focuses on the algebraic property of protomodularity~\cite{NEKEAC}, which in a pointed finitely complete context is equivalent to the validity of the Split Short Five Lemma. The categories $\Grp$ and $\Ab$ are protomodular while $\OrdGrp$ and $\OrdAb$ are not (see~\cite{CM-FM}). It is as if the preorder structure works against protomodularity. This led us to the following question:
\begin{center}
	\textsf{Is there an $\Ord$-enrichment of these categories so that protomodularity is recovered in a lax sense}?
\end{center}
The aim of this work is to give a positive answer to this question for $\OrdAb$; more precisely: there is a preordered structure on morphisms that does work in favour of lax protomodularity in a sense which we explain next.

A finitely complete category $\Xx$ is protomodular~\cite{NEKEAC} when the pullback functors $\alpha^*\colon \Pt_Y(\Xx) \to \Pt_A(\Xx)$ are conservative for any morphism $\alpha\colon A\to Y$. We recall in Theorem~\ref{characterization of protomodularity} well-known equivalent conditions expressed by properties on pullbacks which characterise protomodularity. To obtain a lax version of these equivalences for a finitely complete $\Ord$-enriched category $\C$ with comma objects and $2$-pullbacks, we replace pullbacks conveniently with comma objects or $2$-pullbacks and the pullback change-of-base functors with comma object change-of-base functors. Since comma objects are defined on ordered pairs of morphisms, we have two possible change-of-base functors, given a morphism $\alpha\colon A\to Y$: the \emph{vertical} comma object functor $V_{\alpha}\colon \Pt_Y(\C)\to \Pt_A(\C)$ and the \emph{horizontal} comma object functor $H_{\alpha}\colon \Pt_Y(\C)\to \Pt_A(\C)$. We show that the lax version of the equivalences in Theorem~\ref{characterization of protomodularity} hold for $\C$; see Theorem~\ref{V conservative iff |1|2|}, Corollary~\ref{H conservative iff |1|2|}, Theorem~\ref{V conservative on monos <=> (pi_2,s) jse} and Corollary~\ref{H conservative on monos <=> (pi_1,s) jse}.

Due to the above characterisations, we propose to call a finitely complete $\Ord$-enriched category $\C$ which admits comma objects and $2$-pullbacks \defn{lax protomodular} when the comma object functor $V_{\alpha}$ is conservative for any morphism $\alpha$ in $\C$. We shall call $\C$ \defn{colax protomodular} when $\C^\co$ is lax protomodular, that is, when $H_{\alpha}$ is conservative for any morphism $\alpha$ in $\C$.
The interplay between the different ingredients encoded on protomodularity is worked out in Section \ref{Ord-protomodularity}.

It turns out that $\OrdAb$, equipped with the $\Ord$-enrichment mentioned above, only admits the weaker notion of \emph{precomma objects}, not all comma objects. For this reason it fails to be lax protomodular. However, if we focus instead on precomma objects, vertical (and horizontal) precomma object functors and pullbacks, the equivalences above still hold in this weaker context. We call a finitely complete $\Ord$-enriched category $\C$ with precomma objects \defn{lax preprotomodular} (resp. \defn{colax preprotomodular}) when the precomma object functor $V_{\alpha}$ (resp. $H_\alpha$) is conservative for any morphism $\alpha$ in $\C$. We show in Section~\ref{The category of preordered abelian groups OrdAb} that $\OOrdAb$ is lax preprotomodular but not colax preprotomodular.

In addition, we analyse the behaviour of two different factorization systems in $\OOrdAb$, which correspond to (bijective on objects, fully faithful) and (surjective on objects, monic and fully faithful) factorizations, showing that they are not pullback-stable, and so $\OOrdAb$ fails to be $\Ord$-regular.

\section{The $\Ord$-enriched change-of-base functors}\label{The comma object functors}

In the following $\C$ denotes a finitely complete category enriched in the category $\Ord$ of preordered sets and monotone maps. Recall that a \emph{preorder} is a reflexive and transitive relation. This means that for any objects $X$ and $Y$ of $\C$, $\C(X,Y)$ is equipped with a preorder such that (pre)composition preserves it. We will denote this preorder of morphisms by $\preccurlyeq$. If we consider in $\C$ the reverse preorder we obtain again an $\Ord$-enriched category which we denote, as usual, by $\C^\co$.

A morphism $f\colon X\to Y$ is said to be \defn{fully faithful} when: given morphisms $a,a'\colon A\to X$ such that $fa \preccurlyeq fa'$, then $a\preccurlyeq a'$; equivalently, $fa\preccurlyeq fa'$ if and only if $a\preccurlyeq a'$.

Given an ordered pair of morphisms $(f\colon X\to Y, g\colon Z\to Y)$ in $\C$ with common codomain, the \defn{(strict) comma object} of $(f,g)$ is defined by an object $C$ and morphisms $c_1\colon C\to X$, $c_2\colon C\to Z$ such that
\begin{itemize}
	\item[(C1)] $fc_1\preccurlyeq gc_2$;
	\item[(C2)] it has the universal property: given morphisms $\alpha\colon A\to X$ and $\beta\colon A\to Z$ with $f\alpha\preccurlyeq g\beta$, there exists a unique morphism $\lambda\colon A\to C$ such that $c_1\lambda =\alpha$ and $c_2\lambda = \beta$, as in diagram \eqref{comma f/g};
	\item[(C3)] for morphisms $\alpha,\alpha'\colon A\to X$, $\beta,\beta'\colon A\to Z$ such that $f\alpha\prec g\beta$, $f\alpha'\prec g\beta'$, $\alpha\prec \alpha'$ and $\beta\prec \beta'$, the corresponding unique morphisms $\lambda,\lambda'\colon A\to C$ verify $\lambda\prec \lambda'$;
\end{itemize}

\begin{equation}
\label{comma f/g}
\vcenter{\xymatrix@=30pt{ A \ar@(d,l)[ddr]_-{\alpha} \ar@(u,r)[drr]^-{\beta} \ar@{.>}[dr]^-{\lambda} \\
								 & C \ar[r]^-{c_2} \ar[d]_-{c_1} \ar@{}[dr]|-{\preccurlyeq} & Z \ar[d]^-g \\
								 & X \ar[r]_-f & Y.}}
\end{equation}
The comma object of $(f,g)$ will be denoted by $f/g$, its ``projections'' by $\pi_1\colon$ $f/g\to X$ and $\pi_2\colon f/g\to Z$, and the induced morphism as above by $\lambda=\la \alpha,\beta\ra$. Note that, if $\C$ admits\linebreak 2-products then condition (C3) is equivalent to the full faithfulness of the morphism\linebreak $\la c_1,c_2\ra\colon C\to X\times Z$.

We call a construction as above the \defn{precomma object} of $(f,g)$ when only conditions (C1) and (C2) are required to hold.

Given a comma object diagram as \eqref{comma f/g}, if $fc_1=gc_2$, then it is easy to check that $f/g$ is the 2-pullback of $(f,g)$; similarly the precomma object of $(f,g)$ coincides with the pullback of $(f,g)$.

From now on, $\C$ will denote a finitely complete category enriched in $\Ord$ which admits (pre)comma objects. Our main example $\OOrdAb$ is such a category (see Section~\ref{The category of preordered abelian groups OrdAb}).

Given an object $Y$, as usual we denote by $\C/Y$ the slice category of $\C$ over $Y$. Our main change-of-base functors will be defined on points over $Y$. Here by \emph{point over $Y$} we mean a morphism from the terminal object $1_Y\colon Y\to Y$ into an arbitrary object $f\colon X\to Y$ of $\C/Y$; that is, a $\C$-morphism $s\colon Y\to X$ so that $fs=1_Y$. Hence a point in $\C$ is given by a split epimorphism $f\colon X\to Y$ with a chosen splitting $s\colon Y\to X$. We denote by $\Pt_Y(\C)$ the category of points over $Y$ in $\C$.

Given a morphism $\alpha\colon A\to Y$ in $\C$, we can define two possible functors by taking comma objects along points. For a point $\xymatrix{X \ar@<2pt>[r]^-f & Y \ar@<2pt>[l]^-s}$, we form the comma objects
\[
\xymatrix@=30pt{A \ar@(d,l)[ddr]_-{1_A} \ar@(u,r)[drr]^-{s\alpha} \ar@{.>}[dr]^-{\langle 1_A,s\alpha\rangle} &&&& A \ar@(d,l)[ddr]_-{s\alpha} \ar@(u,r)[drr]^-{1_A} \ar@{.>}[dr]^-{\langle s\alpha,1_A\rangle} \\
								 & \alpha/f \ar[r]^-{\pi_2} \ar[d]_-{\pi_1} \ar@{}[dr]|-{\preccurlyeq} & X \ar@<-2pt>[d]_-f&\mbox{ and }&& f/\alpha \ar[r]^-{\pi_2} \ar[d]_-{\pi_1} \ar@{}[dr]|-{\preccurlyeq} & A \ar[d]^-\alpha & \\
								\mbox{(1.ii)} & A \ar[r]_-\alpha & Y, \ar@<-2pt>[u]_-s&& & X \ar@<2pt>[r]^-f & Y. \ar@<2pt>[l]^-s&\mbox{(1.iii)}}
\]
These constructions define respectively the \emph{vertical} and \emph{horizontal comma objects change-of-base functors}:

$$\begin{array}{rrcl}
	V_{\alpha}: & \Pt_Y(\C) & \longrightarrow & \Pt_A(\C) \\
	& \xymatrix@C=40pt@R=30pt{Z \ar[dr]^-{\gamma} \ar@<-6pt>[ddr]_g \\ & X \ar[d]_-f \\ & Y \ar@<-4pt>[u]_-s \ar@<2pt>[uul]} & \longmapsto &
	\xymatrix@C=40pt@R=30pt{\alpha/g \ar[dr]^-{V_{\alpha}(\gamma)} \ar@<-6pt>[ddr] \\ & \alpha/f \ar[d]_-{\pi_1} \\ & A, \ar@<2pt>[uul] \ar@<-4pt>[u]_-{\la 1_A,s\alpha\ra}}
\end{array}
$$
where $V_{\alpha}(\gamma)$ is induced by the universal property of $\alpha/f$, and

$$
\begin{array}{rrcl}
	H_{\alpha}: & \Pt_Y(\C) & \longrightarrow & \Pt_A(\C) \\
	& \xymatrix@=40pt{Z \ar[dr]_-{\gamma} \ar@<6pt>[drr]^g \\ & X \ar@<2pt>[r]^(.4)f & Y \ar@<-2pt>[ull] \ar@<2pt>[l]^(.6)s} & \longmapsto &
	\xymatrix@=40pt{g/\alpha \ar[dr]_-{H_{\alpha}(\gamma)} \ar@<6pt>[drr] \\ & f/\alpha \ar@<2pt>[r]^(.45){\pi_2} & A,
	\ar@<-2pt>[ull] \ar@<2pt>[l]^-{\la s\alpha, 1_A\ra}}
\end{array}
$$
where $H_{\alpha}(\gamma)$ is induced by the universal property of $f/\alpha$.\\

As in any $\Ord$-enriched category, the notion of adjoint pair of morphisms allows us to consider the following special points, that play special roles in comma objects and consequently in the change-of-base functors, as explained next. A point $(f,s)$ is called \emph{rali} (short for right adjoint, left inverse) when $sf\preccurlyeq 1_X$; it is called \emph{lali} (short for left adjoint, left inverse) when $1_X\preccurlyeq sf$. We denote by $\Ptr_Y(\C)$ (resp. $\Ptl_Y(\C)$) the category of rali (resp. lali) points over $Y$.

\begin{remark}
\label{remark comma objs of rali/lali is rali/lali}
It is easy to check that, in diagram (1.ii), with $\xymatrix{X \ar@<2pt>[r]^-f & Y \ar@<2pt>[l]^-s}$ a rali, also the point
 $\xymatrix{\alpha/f \ar@<2pt>[r]^-{\pi_1} & A \ar@<2pt>[l]^-{\langle 1_A,s\alpha\rangle}}$ is a rali since $\langle \pi_1, \pi_2\rangle\colon \alpha/f\to A \times X$ is fully faithful, $\pi_1\langle 1_A,s\alpha\rangle \pi_1=\pi_1$ and $\pi_2\langle 1_A,s\alpha\rangle \pi_1=s\alpha\pi_1\preccurlyeq sf\pi_2\preccurlyeq \pi_2$.
 Analogously, if $\xymatrix{X \ar@<2pt>[r]^-f & Y \ar@<2pt>[l]^-s}$ is a lali, in diagram (1.iii) the point $\xymatrix{f/\alpha \ar@<2pt>[r]^-{\pi_2} & A \ar@<2pt>[l]^-{\langle s\alpha,1_A\rangle}}$
	is lali.
In particular, for any morphism $\alpha\colon A\to Y$, the point $\xymatrix{\alpha/1_Y\ar@<2pt>[r]^-{\pi_1} & A \ar@<2pt>[l]^-{\langle 1_A,\alpha\rangle}}$ is always rali, while the point $\xymatrix{1_Y/\alpha \ar@<2pt>[r]^-{\pi_2} & A \ar@<2pt>[l]^-{\langle \alpha,1_A\rangle}}$ is always lali.\\

Therefore, \emph{for any morphism $\alpha\colon A\to Y$, $V_{\alpha}$ (co)restricts to a functor on rali points; similarly, $H_{\alpha}$ (co)restricts to a functor on lali points.}
\end{remark}

\section{Lax protomodularity}\label{Ord-protomodularity}

Several algebraic properties in a category $\Xx$ with pullbacks can be expressed by properties of the pullback functors $\alpha^* \colon \Pt_Y(\Xx) \to \Pt_A(\Xx)$, for morphisms $\alpha \colon A \to Y$. One of the key results in this direction is the following (a proof can be found e.g. in \cite{BB}):

\begin{theorem} \label{characterization of protomodularity}
For a category $\Xx$ with pullbacks, the following conditions are equivalent:
\begin{itemize}
\item[(i)] the pullback functors $\alpha^* \colon \Pt_Y(\Xx) \to \Pt_A(\Xx)$ are conservative for every $\alpha \colon A \to Y$;
\item[(ii)] for any commutative diagram of points
\begin{equation}\label{diag:ii}
\xymatrix@=25pt{ \cdot \ar@<-2pt>[d] \ar[r] \ar@{}[dr]|-{\caixa{\mathrm{A}}} & \cdot \ar@<-2pt>[d] \ar[r] \ar@{}[dr]|-{\caixa{\mathrm{B}}} & \cdot \ar@<-2pt>[d] \\
	\cdot \ar@<-2pt>[u] \ar[r] & \cdot \ar@<-2pt>[u] \ar[r] & \cdot \ar@<-2pt>[u]}
\end{equation}
(also commuting with the upward sections) where \caixa{A} and \caixa{A}\caixa{B} are pullbacks, then \caixa{B} is also a pullback;
\item[(iii)] for any pullback of a point $(f,s)$ along an arbitrary morphism
\begin{equation}\label{diag:iii}
\xymatrix@=25pt{ A\times_Y X \ar@<-2pt>[d]_-{\pi_1} \ar[r]^-{\pi_2} \ophalfsplitpullback & X \ar@<-2pt>[d]_-f  \\
							   A \ar@<-2pt>[u] \ar[r] & Y, \ar@<-2pt>[u]_-s }
\end{equation}
the pair $(\pi_2,s)$ is jointly extremally epimorphic.
\end{itemize}
\end{theorem}

A finitely complete category $\Xx$ is called \defn{protomodular}\footnotemark\footnotetext{The original definition only asks for $\Xx$ to admit the existence of pullbacks of points along arbitrary morphisms. When $\Xx$ is such, the equivalences of Theorem~\ref{characterization of protomodularity} still hold.}~\cite{NEKEAC} if the equivalent conditions of the previous theorem hold in $\Xx$.

\begin{example}\label{exa:prot}
Some examples of protomodular categories are (see~\cite{BB} for more examples):\\
\begin{itemize}
	\item The variety $\Grp$ of groups and, more generally, of $\Omega$-groups (the corresponding theory has, among its operations, a unique constant and the group operations).  In fact, in \cite{BJ} the varieties which form a protomodular category were characterized as those for which there exist $n\in\mathbb{N}$ and
\begin{itemize}
\item constants $e_1,\dots,e_n$;
\item binary operations $\alpha_1,\dots,\alpha_n$ such that $\alpha_i(x,x)=e_i$ for all $i=1,\dots,n$;
\item an ($n+1$)-ary operation $\theta$ such that $\theta(\alpha_1(x,y),\dots,\alpha_n(x,y),y)=x$.
\end{itemize}
	\item Any additive category with finite limits.
	\item $\Set^{\op}$ and, more generally, the dual of any elementary topos.
\end{itemize}
\end{example}

If, moreover, the category $\Xx$ is pointed, then it is immediate to see that the conservativeness of all pullback functors $\alpha^* \colon \Pt_Y(\Xx) \to \Pt_A(\Xx)$ is implied by (and hence equivalent to) the conservativeness of the pullback functors induced by the morphisms whose domain is the zero object. This last property is equivalent to the classical Split Short Five Lemma (see, for example, Proposition $3.1.2$ in \cite{BB}). Hence, for pointed categories, \emph{protomodularity is equivalent to the validity of the Split Short Five Lemma}. \\

Our goal now is to replace the protomodularity condition with a lax version of it, where we look at the conservativeness of the comma object functors. We can consider a ``vertical'' version of Theorem~\ref{characterization of protomodularity} where the $\alpha^*$ are replaced by $V_{\alpha}$ and the points are vertical arrows as in diagrams \eqref{diag:ii} and \eqref{diag:iii}. Or we can consider the ``horizontal'' case with $H_{\alpha}$ and where the points appear horizontally in those diagrams. Despite this distinction, the equivalences obtained in $\C$ in one direction give immediately the corresponding results in the other direction when applied to $\C^\co$.

We chose vertical as our ``priority case'' with detailed proofs and simply state the equivalences for the horizontal case. Note that an $\Ord$-enriched category may fulfill the equivalent vertical properties and fail to fulfill the horizontal ones, or vice-versa. We propose the names \emph{lax protomodular category} for an $\Ord$-enriched finitely complete category $\C$ with comma objects and $2$-pullbacks such that the above mentioned equivalences hold in the vertical direction and \emph{colax protomodular category} with respect
to the horizontal direction (see Definition~\ref{Ord-enriched protomodularity}). \\

Replacing pullbacks with comma objects to get similar equivalences as those in Theorem~\ref{characterization of protomodularity} is not straightforward. Properties on pullbacks do not give similar properties on comma objects; e.g. gluing comma objects together does not give a comma object in general. However, there are some well-known properties combining comma objects and 2-pullbacks, whose proof can be found, for instance, in \cite{Vassilis}:
\begin{lemma}
\label{2 comma. 12 comma <=> 1 pb}
Let $\C$ be a finitely complete $\Ord$-enriched category which admits comma objects. Consider the diagram where the right square is a comma object and the left square is commutative
$$
	\xymatrix{P \ar[d]_-{p_1} \ar[r]^-{p_2} & f/g \ar[d]_-{\pi_1} \ar[r]^-{\pi_2} \ar@{}[dr]|-{\preccurlyeq} & Z \ar[d]^-g \\
						X' \ar[r]_-x & X \ar[r]_-f &  Y.}
$$
The outer rectangle is a comma object if and only if the left square is a $2$-pullback.
\end{lemma}

A similar result holds by stacking squares vertically:
\begin{lemma}
\label{2 comma. 12 comma <=> 1 pb, vertical version}
Let $\C$ be a finitely complete $\Ord$-enriched category which admits comma objects. Consider the diagram where the bottom square is a comma object and the top square is commutative
$$
	\xymatrix{P \ar[d]_-{p_1} \ar[r]^-{p_2} & Z' \ar[d]^-{z} \\
						f/g \ar[d]_-{\pi_1} \ar[r]^-{\pi_2} \ar@{}[dr]|-{\preccurlyeq} & Z \ar[d]^-g \\
						X \ar[r]_-f &  Y.}
$$
The outer rectangle is a comma object if and only if the top square is a $2$-pullback.
\end{lemma}

We now prove that when we replace the pullback functor with a comma object functor and the pullbacks with comma objects we obtain enriched versions of the  equivalence Theorem~\ref{characterization of protomodularity} (i) $\Leftrightarrow$ (ii).

\begin{theorem}\label{V conservative iff |1|2|} Let $\C$ be a finitely complete $\Ord$-enriched category which admits comma objects and $2$-pullbacks. The following statements are equivalent:
\begin{itemize}
\item[(i)] $V_{\alpha}\colon \Pt_Y(\C)\to \Pt_A(\C)$ is conservative for any morphism $\alpha\colon A\to Y$.
\item[(ii)] In any diagram where the left square \caixa{1} and total rectangle \caixa{1}\caixa{2} are comma diagrams and the right square \caixa{2} is commutative
\begin{equation}\label{diag:comma}
	\xymatrix@R=30pt@C=40pt{\alpha/f \ar@{}[dr]|-{\preccurlyeq} \ar@<-2pt>[d]_-{\pi_1} \ar[r]^-{\pi_2} &
													X \ar@<-2pt>[d]_-{f} \ar[r]^-{\chi}  & 	U \ar@<-2pt>[d]_-g \\
													A \ar@<-2pt>[u]_-{i_1} \ar[r]_-{\alpha}^-{\mbox{\tiny \caixa{1}}} &
													Y \ar@<-2pt>[u]_-s \ar[r]_-{\beta}^-{\mbox{\caixa{2}}} &  V, \ar@<-2pt>[u]_-t}
\end{equation}
\begin{center}\tiny
		($\alpha\pi_1\preccurlyeq f\pi_2$, $i_1\stackrel{\mathrm{(1.ii)}}{=}\la 1_A,s\alpha\ra$, $\;fs=1_Y$, $\;gt=1_V$, $\;\beta f=g\chi$, $\;\chi s=t\beta$),
\end{center}
the right square \caixa{2} is a $2$-pullback.
\end{itemize}
\end{theorem}

\begin{proof} (i) $\Rightarrow$ (ii). In the following diagram
$$
	\xymatrix@R=30pt@C=40pt{\alpha/1_Y \ar[r]^-{\rho_2} \pullback \ar@{.>}[d]_-{\sigma} \ar@(l,l)[dd]_-{\rho_1} & Y \ar[d]^-s \ar@(r,r)@{=}[dd]^-{1_Y}\\
						\alpha/f \ar@{}[dr]|-{\preccurlyeq} \ar[r]_-{\pi_2}^-{\mbox{\caixa{3}}} \ar[d]_-{\pi_1} & X \ar[d]^-f \\
						A \ar[r]_-{\alpha}^-{\mbox{\caixa{1}}} & Y}
$$
we get an induced morphism $\sigma\colon \alpha/1_Y \to \alpha/f$ such that the bottom square and total rectangle are comma objects; thus the top square \caixa{3} is a $2$-pullback (Lemma~\ref{2 comma. 12 comma <=> 1 pb, vertical version}). Similarly, we get an induced morphism $\varphi\colon \alpha/f\to \alpha/1_Y$
$$
	\xymatrix@R=30pt@C=40pt{\alpha/f \ar[r]^-{\pi_2} \pullback \ar@{.>}[d]_-{\varphi} \ar@(l,l)[dd]_-{\pi_1} & X \ar[d]^-f \ar@(r,r)[dd]^-{f}\\
						\alpha/1_Y \ar@{}[dr]|-{\preccurlyeq} \ar[r]_-{\rho_2}^-{\mbox{\caixa{5}}} \ar[d]_-{\rho_1} & Y \ar@{=}[d]^-{1_Y} \\
						A \ar[r]_-{\alpha}^-{\mbox{\caixa{4}}} & Y}
$$
such that \caixa{5} is a $2$-pullback. It is easy to check that $\varphi\sigma=1_{\alpha/1_Y}$.

Next we prove that the following diagram
$$
	\xymatrix@R=30pt@C=40pt{ \alpha/1_Y \ar[r]^-{\rho_2} \ar[d]_-{\rho_1} \ar@{}[dr]|-{\preccurlyeq} & Y\ar[r]^-{\beta} \ar@{=}[d]^-{1_Y}  & V \ar@{=}[d]^-{1_V} \\
							A \ar[r]_-{\alpha}^-{\mbox{\caixa{4}}} & Y \ar[r]_-{\beta}^-{\mbox{\caixa{6}}} & V}		
$$
is a comma object diagram. Indeed:
\begin{itemize}
	\item[(C1)] $\beta\alpha\rho_1\preccurlyeq \beta\rho_2$;
	\item[(C2)] if $m\colon M\to A$ and $n\colon M\to V$ are such that $\beta\alpha m \preccurlyeq n$, then $\beta\alpha m \preccurlyeq gtn$. We get
	an induced morphism $\lambda\colon M\to \alpha/f$
	$$
	\xymatrix@R=30pt@C=40pt{M \ar@{.>}[dr]^-{\lambda} \ar@(d,l)[ddr]_-m \ar@(r,u)[drrr]^-{tn} \\
						& \alpha/f \ar@{}[dr]|-{\preccurlyeq} \ar[r]^-{\pi_2} \ar@<-2pt>[d]_-{\pi_1} & X \ar@<-2pt>[d]_-f \ar[r]^-{\chi} &
						U \ar@<-2pt>[d]_-g\\
						& A \ar@<-2pt>[u] \ar[r]_-{\alpha}^-{\mbox{\caixa{1}}} & Y \ar[r]_-{\beta}^-{\mbox{\caixa{2}}} \ar@<-2pt>[u]_-s & V \ar@<-2pt>[u]_-t }
	$$
	such that $\pi_1\lambda=m$ and $\chi\pi_2\lambda=tn$. We consider the morphism $\varphi\lambda\colon M\to \alpha/1_Y$ which is such that $\rho_1\varphi\lambda=\pi_1\lambda=m$ and $\beta\rho_2\varphi\lambda=\beta f\pi_2 \lambda=g\chi\pi_2\lambda=gtn=n$. To prove the uniqueness of such a morphism, suppose that $\xi\colon M\to \alpha/1_Y$ is such that $\rho_1\xi=m$ and $\beta\rho_2\xi=n$. Then $t\beta\rho_2\xi=tn$ or, equivalently, $\chi s\rho_2\xi=tn$, from which we get $\chi\pi_2\sigma\xi=tn$. Since also, $\pi_1\sigma\xi=\pi_1\xi=m$, we conclude that $\lambda=\sigma\xi$, from the universal property of the comma object of $(\beta\alpha,g)$. It follows that $\varphi\lambda=\varphi\sigma\xi=\xi$;
	\item[(C3)] consider morphisms $c,c'\colon B\to A, d,d'\colon B\to V$ such that $\beta\alpha c\prec d$, $\beta\alpha c'\prec d'$, $c\prec c'$ and $d\prec d'$. Let $b=\la c,d\ra,b'=\la c',d'\ra$ be the induced morphisms from $B$ to $\alpha/1_Y$. With respect to the comma object diagram \caixa{1}\caixa{2}, we have morphisms $c,c',td,td'$ such that $\beta\alpha c\prec gtd$, $\beta\alpha c'\prec gtd'$, $c\prec c'$ and $td\prec td'$. Thus the induced morphisms $\sigma b,\sigma b'\colon B\to \alpha/f$ are such that $\sigma b\prec \sigma b'$. Then $b=\varphi\sigma b \preccurlyeq \varphi \sigma b'=b'$.
\end{itemize}

In the following diagram
$$
	\xymatrix@R=30pt@C=40pt{\alpha/f \ophalfsplitpullback \ar[r]^-{\pi_2} \ar@<-2pt>[d]_-{\varphi} \ar@(l,l)[dd]_-{\pi_1} &
													X \ar[r]^-{\chi} \ar@<-2pt>[d]_-f &
													U \ar@<-2pt>[d]_-g\\
													\alpha/1_Y \ar[r]_-{\rho_2}^-{\mbox{\caixa{5}}} \ar[d]_-{\rho_1} \ar@<-2pt>[u]_-{\sigma} &
													Y \ar@{=}[d] \ar[r]_-{\beta}^-{\mbox{\caixa{2}}} \ar@<-2pt>[u]_-s &
													V \ar@{=}[d] \ar@<-2pt>[u]_-t \\
													A \ar[r]_-{\alpha}^-{\mbox{\caixa{4}}} & Y \ar[r]_-{\beta}^-{\mbox{\caixa{6}}} & V,}
$$
the bottom rectangle and the total diagram are comma objects, so that the top rectangle is a $2$-pullback (Lemma~\ref{2 comma. 12 comma <=> 1 pb, vertical version}).

Next, we take the $2$-pullback of $\beta$ and $g$ to get the following diagram
$$
	\xymatrix@R=30pt@C=50pt{\alpha/f \ar[rr]^-{\pi_2} \ar@<-2pt>[dd]_-{\varphi} \ar@{.>}[drrr]^-{\pi} & &
													X \ar[rr]^-{\chi} \ar@<-2pt>[dd]_-{f} \ar@{.>}[dr]^-{\la f,\chi\ra} & &
													\ar@<-2pt>[dd]_-g U \\
													& & & P\ar[ur]^-{\beta'} \ar@<-2pt>[dl]_-{g'}\\
													\alpha/1_Y \ar[rr]_-{\rho_2} \ar@<-2pt>[uu]_-{\sigma} & &
													Y \ar[rr]_-{\beta} \ar@<-2pt>[uu]_-s  \ar@<-2pt>[ur]_-{t'} & &
													V. \ar@<-2pt>[uu]_-t
	}
$$
Since the whole rectangle is a $2$-pullback, the bottom left quadrangle is a $2$-pullback; let us call it \caixa{7}. The diagram composed by \caixa{4} on the bottom and \caixa{7} on top is a comma object diagram (Lemma~\ref{2 comma. 12 comma <=> 1 pb, vertical version}). This is the front face in
$$
\xymatrix@C=40pt@R=40pt{\alpha/f \ar@<-2pt>[ddr]_-{\pi_1} \ar[rr]^-{\pi_2}
								\ar@{->}[dr]^(.6){V_{\alpha}(\la f,\chi\ra)=1_{\alpha/f}}|-{\cong} & &
								X \ar@<-6pt>[ddr]_(.3){f} \ar[dr]^-{\la f,\chi\ra} \\
								& \alpha/f \ar@<-2pt>[d]_(.4){\pi_1} \ar@<2pt>[rr]^(.4){\pi} & &
								\;P\; \ar@<-2pt>[d]_(.4){g'} \\
								& A \ar@<-2pt>[u]_(0.6){i_1} \ar@<-2pt>[uul]_-{i_1} \ar[rr]_-{\alpha} & &
								Y; \ar@<2pt>[uul]_(.65)s \ar@<-2pt>[u]_(.6){t'}}
$$
the left point being $(\rho_1\varphi, \sigma\la 1_A, \alpha\ra) = (\pi_1,i_1)$. Note that, $V_{\alpha}(\la f,\chi\ra)$ $=1_{\alpha/f}$, so that $\la f,\chi\ra$ is an isomorphism by assumption. This proves that \caixa{2} is indeed a $2$-pullback.\\
(ii) $\Rightarrow$ (i). Consider an arbitrary morphism $\alpha\colon A\to Y$ and a morphism $\gamma\colon (f,s) \to (g,t)$ in $\Pt_Y(\C)$. Suppose that $V_{\alpha}(\gamma)$ is an isomorphism
$$
\xymatrix@C=40pt@R=40pt{\alpha/f \ar@<-2pt>[ddr]_-{\pi_1} \ar[rr]^-{\pi_2} \ar@{->}[dr]^(.6){V_{\alpha}(\gamma)}|-{\cong} & &
								X \ar@<-4pt>[ddr]_(.3){f} \ar@<2pt>[dr]^-{\gamma} \\
								& \alpha/g \ar@<-2pt>[d]_(.4){\rho_1} \ar[rr]^(.4){\rho_2} & & U \ar@<-2pt>[d]_(.4){g} \\
								& A \ar[rr]_-{\alpha} \ar@<-2pt>[u]_(.6){j_1} \ar@<-2pt>[uul]_-{i_1} & & Y; \ar[uul]_(.65)s \ar@<-2pt>[u]_(.6)t}
$$
here $i_1=\la 1_A,s\alpha\ra$ and $j_1=\la 1_A, t\alpha\ra$. The diagram
$$
	\xymatrix@R=30pt@C=40pt{\alpha/f \ar@{}[dr]|-{\preccurlyeq} \ar@<-2pt>[d]_-{\pi_1} \ar[r]^-{\pi_2} &
													X \ar@<-2pt>[d]_-{f} \ar[r]^-{\gamma}  \ar@{}[dr]|-{\mbox{\caixa{*}}}&
													U \ar@<-2pt>[d]_-g \\
													A \ar@<-2pt>[u]_-{i_1} \ar[r]_-{\alpha} & Y \ar@<-2pt>[u]_-s \ar@{=}[r]_-{1_Y} &  Y \ar@<-2pt>[u]_-t}
$$
is of the type \eqref{diag:comma}. The whole rectangle is a comma object diagram because it is the same as
$$
	\xymatrix@R=30pt@C=40pt{\alpha/f  \pullback \ar[d]_-{\pi_1} \ar[r]^-{V_{\alpha}(\gamma)}|-{\cong} &
													\alpha/g \ar@{}[dr]|-{\preccurlyeq} \ar[d]_-{\rho_1} \ar[r]^-{\rho_2} &
													U \ar[d]^-g \\
													A \ar@{=}[r]_-{1_A} & A \ar[r]_-{\alpha} &  Y,}
$$
where the left commutative square is obviously a 2-pullback (see Lemma~\ref{2 comma. 12 comma <=> 1 pb}). By assumption \caixa{*} is a $2$-pullback; thus $\gamma$ is an isomorphism.
\end{proof}

The horizontal version of the previous result is stated next.

\begin{corollary}\label{H conservative iff |1|2|} Let $\C$ be a finitely complete $\Ord$-enriched category which admits comma objects and $2$-pullbacks. The following statements are equivalent:
\begin{itemize}
\item[(i)] $H_{\alpha}\colon \Pt_Y(\C)\to \Pt_A(\C)$ is conservative for any morphism $\alpha\colon A\to Y$.
\item[(ii)] In any diagram where the top square \caixa{1} and total rectangle are comma diagrams and the bottom square \caixa{2} is commutative
$$
	\xymatrix@R=30pt@C=40pt{f/\alpha \ar@{}[dr]|-{\preccurlyeq} \ar@<2pt>[r]^-{\pi_2} \ar[d]_-{\pi_1} &
													A \ar@<2pt>[l]^-{i_2} \ar[d]^-{\alpha}_-{\mbox{\tiny \caixa{1}}}\\
													X \ar@<2pt>[r]^-f \ar[d]_-{\chi} & Y \ar[d]^-{\beta}_-{\mbox{\tiny \caixa{2}}} \ar@<2pt>[l]^-s  \\
													U \ar@<2pt>[r]^-g  &  V, \ar@<2pt>[l]^-t}
$$
\begin{center}\tiny
		($f\pi_1\preccurlyeq \alpha\pi_2$, $i_2\stackrel{\mathrm{(1.iii)}}{=}\la s\alpha,1_A\ra$, $\;fs=1_Y$, $\;gt=1_V$, $\;\beta f=g\chi$, $\;\chi s=t\beta$),
\end{center}
the bottom square \caixa{2} is a 2-pullback.
\end{itemize}
\end{corollary}

\begin{corollary}\label{proto => 2-proto} Let $\Xx$ be a protomodular category. Then any $\Ord$-enrichment $\X$ of $\Xx$ that admits comma objects and $2$-pullbacks fulfills the equivalent conditions of Theorem~\ref{V conservative iff |1|2|} and of Corollary~\ref{H conservative iff |1|2|}.
\end{corollary}
\begin{proof} Note that pullbacks coincide with $2$-pullbacks in this setting. Consider a diagram as in Theorem~\ref{V conservative iff |1|2|}.(ii). Following its proof above, we deduce that \caixa{5}\caixa{2} is a ($2$-)pullback. Since \caixa{5} is also a ($2$-)pullback, we conclude that \caixa{2} is a ($2$-)pullback from protomodularity (Theorem \ref{characterization of protomodularity}).
\end{proof}

\begin{remark} If $\C$ is a finitely complete $\Ord$-enriched category $\C$ which admits comma objects, then the equivalences of Theorem~\ref{V conservative iff |1|2|} and Corollary~\ref{H conservative iff |1|2|} still hold when in the diagrams (ii) we conclude that \caixa{2} is a ($1$-)pullback. Indeed, in the proof of the implication (i) $\Rightarrow$ (ii) in Theorem~\ref{V conservative iff |1|2|}, we can take the pullback of $\beta$ and $g$ and still conclude that \caixa{7} is a $2$-pullback. The rest of that proof is the same. For (ii) $\Rightarrow$ (i) in Theorem~\ref{V conservative iff |1|2|}, we get that \caixa{*} is a pullback and can still conclude that $\gamma$ is an isomorphism.
\end{remark}

In order to prove that the remaining condition defining protomodular categories is equivalent to the other two, the following well-known fact is used: the conservativeness property of a left exact functor $F$ is equivalent to its conservativeness on monomorphisms. A functor $F$ is said to be conservative on monomorphisms if, for every monomorphism $f$, if $F(f)$ is an isomorphism, then so is $f$. Any pullback
functor $\alpha^* \colon \Pt_Y(\Xx)\to \Pt_A(\Xx)$ preserves finite limits, as soon as the base category $\Xx$ admits pullbacks along split epimorphisms. In the enriched context, such a property fails
to hold for the comma objects functors. In fact, given a morphism $\alpha\colon A\to Y$ and the terminal object $(1_Y,1_Y)$ of $\Pt_Y(\C)$, the comma object $\alpha/1_Y$ does
not necessarily give rise to the terminal object $(1_A,1_A)$ of $\Pt_A(\C)$ (nor does $1_Y/\alpha$). Despite this setback, we still obtain a similar result with respect to
the conservativeness of $V_{\alpha}$ and $H_{\alpha}$.

\begin{proposition}\label{V_alpha conservative iff conservative on monos} Let $\C$ be a finitely complete $\Ord$-enriched category which admits comma objects. The following statements are equivalent:
\begin{itemize}
	\item[(i)] $V_{\alpha}\colon \Pt_Y(\C)\to \Pt_A(\C)$ is conservative, for any morphism $\alpha\colon A\to Y$.
	\item[(ii)] $V_{\alpha}\colon \Pt_Y(\C)\to \Pt_A(\C)$ is conservative on monomorphisms, for any morphism $\alpha\colon A\to Y$.
\end{itemize}
\end{proposition}

\begin{proof} (i) $\Rightarrow$ (ii). Obvious \\
(ii) $\Rightarrow$ (i). Consider an arbitrary morphism $\alpha\colon A\to Y$ and a morphism $\gamma\colon (f,s) \to (g,t)$ in $\Pt_Y(\C)$. Suppose that $V_{\alpha}(\gamma)$ is an isomorphism
$$
\xymatrix@C=40pt@R=40pt{\alpha/f \ar@<-2pt>[ddr]_-{\pi_1} \ar[rr]^-{\pi_2} \ar@{->}[dr]^(.6){V_{\alpha}(\gamma)}|-{\cong} & &
								X \ar@<-4pt>[ddr]_(.3){f} \ar@<2pt>[dr]^-{\gamma} \\
								& \alpha/g \ar@<-2pt>[d]_(.4){\rho_1} \ar[rr]^(.4){\rho_2} & & U \ar@<-2pt>[d]_(.4){g} \\
								& A \ar[rr]_-{\alpha} \ar@<-2pt>[u]_(.6){j_1} \ar@<-2pt>[uul]_-{i_1} & & Y; \ar[uul]_(.65)s \ar@<-2pt>[u]_(.6)t}
$$
here $i_1=\la 1_A,s\alpha\ra$ and $j_1=\la 1_A,t\alpha\ra$.
From Lemma~\ref{2 comma. 12 comma <=> 1 pb, vertical version}, the top quadrangle is a $2$-pullback. Taking kernel pairs vertically we obtain a diagram where both downward squares are $2$-pullbacks
$$
\xymatrix@=30pt{
				\alpha/f \ar[r]^-{\pi} \ar@{=}[d] \pullback & \Eq(\gamma) \ar@<-5pt>[d]_-{\gamma_1} \ar@<5pt>[d]^-{\gamma_2} \\
				\alpha/f \ar[r]_-{\pi_2} & X; \ar[u]|-d}
$$
here $d=\la 1_X,1_X\ra$. Combining the $2$-pullback with first projections and the comma object of $(\alpha,f)$, we get a comma object (Lemma~\ref{2 comma. 12 comma <=> 1 pb, vertical version}) which is the front face in
$$
\xymatrix@C=40pt@R=40pt{\alpha/f \ar@<-2pt>[ddr]_-{\pi_1} \ar[rr]^-{\pi_2} \ar@{=}[dr]^(.6){V_{\alpha}(d)} & &
								X \ar@<-4pt>[ddr]_(.3){f} \ar@<2pt>@{ >->}[dr]^-{d} \\
								& \alpha/f \ar@<-2pt>[d]_(.4){\pi_1} \ar[rr]^(.4){\pi} & & \Eq(\gamma) \ar[d]_(.3){f\gamma_1} \\
								& A \ar[rr]_-{\alpha} \ar@<-2pt>[u]_(.6){i_1} \ar@<-2pt>[uul]_-{i_1} & & Y. \ar[uul]_(.65)s \ar@<-4pt>[u]_(.6){ds}}
$$
By assumption, the monomorphism $d$ is an isomorphism. This implies that $\gamma_1=\gamma_2$ and, consequently, $\gamma$ is a monomorphism. Applying our assumption again, we conclude that $\gamma$ is an isomorphism.
\end{proof}

\begin{corollary}\label{H_alpha conservative iff conservative on monos} Let $\C$ be a finitely complete $\Ord$-enriched category which admits comma objects. The following statements are equivalent:
\begin{itemize}
	\item[(i)] $H_{\alpha}\colon \Pt_Y(\C)\to \Pt_A(\C)$ is conservative, for any morphism $\alpha\colon A\to Y$.
	\item[(ii)] $H_{\alpha}\colon \Pt_Y(\C)\to \Pt_A(\C)$ is conservative on monomorphisms, for any morphism $\alpha\colon A\to Y$.
\end{itemize}
\end{corollary}

Thanks to these facts, we get the following:

\begin{theorem}\label{V conservative on monos <=> (pi_2,s) jse} Let $\C$ be a finitely complete $\Ord$-enriched category which admits comma objects. The following statements are equivalent:
\begin{itemize}
	\item[(i)] $V_{\alpha}\colon \Pt_Y(\C)\to \Pt_A(\C)$ is conservative, for any morphism $\alpha\colon A\to Y$.
	\item[(ii)] for any comma object of the form
	$$
	\xymatrix{\alpha/f \ar[r]^-{\pi_2} \ar@<-2pt>[d]_-{\pi_1} \ar@{}[dr]|-{\preccurlyeq} & X \ar@<-2pt>[d]_-f\\
						A \ar@<-2pt>[u]_-{i_1} \ar[r]_-{\alpha} & Y, \ar@<-2pt>[u]_-s}
	$$
	the pair $(\pi_2,s)$ is jointly extremally epimorphic; here $i_1\stackrel{\mathrm{(1.ii)}}{=}\la 1_A,s\alpha\ra$.
\end{itemize}
\end{theorem}

\begin{proof} (i) $\Rightarrow$ (ii). Let $m$ be a monomorphism for which the diagram
$$
\xymatrix{ & M \ar@{ >->}[d]^-m \\
					 \alpha/f \ar[ur]^-p \ar[r]_-{\pi_2} & X & Y \ar[ul]_-{\sigma} \ar[l]^-s}
$$
commutes. We get the following diagram
$$
\xymatrix@C=40pt@R=40pt{\alpha/f \ar@<-2pt>[ddr]_-{\pi_1} \ar[rr]^-{p} \ar@{=}[dr]^-{V_{\alpha}(m)=1_{\alpha/f}} & &
								M \ar@<-2pt>[ddr]_(.3){fm} \ar@<4pt>@{ >->}[dr]^-{m} \\
								& \alpha/f \ar@<-2pt>[d]_(.4){\pi_1} \ar[rr]^(.4){\pi_2} & & X \ar@<-2pt>[d]_(.3){f}\\
								& A \ar[rr]_-{\alpha} \ar@<-2pt>[u]_(.6){i_1} \ar@<-2pt>[uul]_-{i_1} & & Y. \ar@<-2pt>[u]_(.7)s \ar@<-2pt>[uul]_(.65){\sigma}}
$$
Note that the top quadrangle is a $2$-pullback, so that the back face is a comma object diagram (Lemma~\ref{2 comma. 12 comma <=> 1 pb, vertical version}). Since $m$ is a monomorphism and $mpi_1=\pi_2i_1=s\alpha=m\sigma\alpha$, it follows that $pi_1=\sigma\alpha$. From the assumption we conclude that $m$ is an isomorphism.\\
(ii) $\Rightarrow$ (i) Conversely, for any morphism $\alpha\colon A\to Y$, a diagram as above gives a factorisation of $\pi_2$ and $s$ through the monomorphism $m$, which is then an isomorphism by assumption. This proves that $V_{\alpha}$ is conservative on monomorphisms, which is equivalent to being conservative (Proposition~\ref{V_alpha conservative iff conservative on monos}).
\end{proof}

\begin{corollary}\label{H conservative on monos <=> (pi_1,s) jse} Let $\C$ be a finitely complete $\Ord$-enriched category which admits comma objects. The following statements are equivalent:
\begin{itemize}
	\item[(i)] $H_{\alpha}\colon \Pt_Y(\C)\to \Pt_A(\C)$ is conservative, for any morphism $\alpha\colon A\to Y$.
	\item[(ii)] for any comma object of the form
	$$
	\xymatrix{f/\alpha \ar@<2pt>[r]^-{\pi_2} \ar[d]_-{\pi_1} \ar@{}[dr]|-{\preccurlyeq} & A \ar@<2pt>[l]^-{i_2} \ar[d]^-{\alpha}\\
						X \ar@<2pt>[r]^-{f} & Y, \ar@<2pt>[l]^-s}
	$$

	the pair $(\pi_1,s)$ is jointly extremally epimorphic; here $i_2\stackrel{\mathrm{(1.iii)}}{=}\la s\alpha,1_A\ra$.
\end{itemize}
\end{corollary}

Having recovered the lax versions of the equivalences in Theorem \ref{characterization of protomodularity} for $\Ord$-enriched categories, we can now propose the following:

\begin{definition}\label{Ord-enriched protomodularity} A finitely complete $\Ord$-enriched category $\C$ which admits comma objects and $2$-pullbacks is called \defn{lax protomodular} when the comma object functor $V_{\alpha}$ is conservative for any morphism $\alpha\colon A\to Y$ in $\C$. We say that $\C$ is \defn{colax protomodular} when $\C^\co$ is lax protomodular, that is when the comma object functor $H_{\alpha}$ is conservative for any morphism $\alpha\colon A\to Y$ in $\C$.
\end{definition}

Exactly as in the classical case, it is immediate to see that, in a pointed $\Ord$-enriched category $\C$ with comma objects, the conservativeness of all comma object funtors $V_{\alpha}$ is equivalent to the conservativeness of the functors $V_{i_A}$, where $i_A \colon 0 \to A$ is the only arrow from the zero object (and the same holds for the $H$'s). Then, in the pointed context, $\C$ is lax protomodular if and only if the following \emph{lax version of the Split Short Five Lemma} holds:

\begin{theorem} Let $\C$ be a pointed finitely complete $\Ord$-enriched category which admits comma objects and 2-pullbacks. $\C$ is lax protomodular if and only if, given a commutative diagram of split sequences of the form
\[ \xymatrix{ 0/f \ar[d]_{\alpha} \ar[r]^k & X \ar[d]_{\beta} \ar@<-2pt>[r]_f & A \ar[d]^{\gamma} \ar@<-2pt>[l]_s \\
0/f' \ar[r]^k & X \ar@<-2pt>[r]_{f'} & A', \ar@<-2pt>[l]_{s'} } \]
if $\alpha$ and $\gamma$ are isomorphisms, then $\beta$ is, too.
\end{theorem}

A similar result holds for pointed colax protomodular categories.

\begin{remark}
All the results of this section still hold when replacing comma objects with precomma objects and when replacing 2-pullbacks with (1-)pullbacks. We use the same notation $V_\alpha$ and $H_\alpha$ for the precomma object functors. A finitely complete $\Ord$-enriched category $\C$ which admits precomma objects is called \defn{lax preprotomodular} when the precomma object functor $V_{\alpha}$ is conservative for any morphism $\alpha$ in $\C$. We say that $\C$ is \defn{colax preprotomodular} when $\C^\co$ is lax preprotomodular.
\end{remark}

\begin{example}\label{examples}
\begin{enumerate}
	\item According to Corollary~\ref{proto => 2-proto}, an $\Ord$-enriched protomodular category with comma objects and 2-pullbacks is both lax and colax protomodular.
	\item If $\mathcal{T}$ is the theory of a protomodular variety, then the category $\C^{\mathcal{T}}$  of internal $\mathcal{T}$-algebras in $\C$ is also protomodular, and so it is both lax protomodular and colax protomodular, for any compatible $\Ord$-enrichment which admits 2-pullbacks and comma objects.

This applies in particular to $\Ord^{\mathcal T}$.
We point out, however, that, for any algebraic theory $\mathcal{T}$ which contains a Mal'tsev operation, the preorder of any internal $\mathcal T$-algebra in $\Ord$ is symmetric: if $X\in \Ord^{\mathcal T}$ and $p\colon X^3\to X$ is a monotone Mal'tsev operation, then, for $x,y\in X$ with $x\leq y$ one obtains that $y=p(x,x,y)\leq p(x,y,y)=x$. This is the case of every theory $\mathcal{T}$ of a protomodular variety since it has a Mal'tsev operation defined by
\[p(x,y,z)=\theta(\alpha_1(x,y),\dots,\alpha_n(x,y),z)\]
(using the operations described in Example \ref{exa:prot}). Therefore, the obvious $\Ord$-enrichment inherited from $\Ord$ -- the pointwise $\Ord$-enrichment -- will be also symmetric.
	\item The former example raises the question
\begin{center}
	\textsf{Is there a protomodular category with a non-degenerate $\Ord$-enrichment with comma objects and 2-pullbacks}?
\end{center}
(where by degenerate $\Ord$-enrichment we mean one whose hom-sets preorders are symmetric).

Since the dual of an elementary topos is protomodular, one can more specifically ask:
\begin{center}
	\textsf{Is there an elementary topos with a non-degenerate $\Ord$-enrichment admitting cocomma objects and 2-pushouts}?
\end{center}
We have put this question to Peter Johnstone, who collected some interesting results on the subject in \cite{PTJ}. Namely:
\begin{itemize}
\item[--] \emph{The only poset-enrichment of a localic topos over $\Set$ is the discrete one.}
\item[--] \emph{The $\Ord$-enrichment of a topos for which equalizers and exponential adjunctions are $\Ord$-enriched is degenerate.}
\end{itemize}
Moreover, in \cite{PTJ} an example of a topos with a non-degenerate $\Ord$-enrichment is presented, but it does not fulfil our conditions since it does not have cocomma objects.
\item In the next section we will study the behaviour of $\OrdAb$ equipped with a suitable $\Ord$-enrichment. It is not an example of a (co)lax proto\-modular category since it does not admit comma objects (although it admits precomma objects). Hence we pose the more general open question
\begin{center}
	\textsf{Is there any lax protomodular or colax protomodular non-degenerate $\Ord$-enriched category}?
\end{center}
\end{enumerate}
\end{example}

\section{The 2-category $\OOrdAb$ of preordered abelian groups}\label{The category of preordered abelian groups OrdAb}

Both the categories $\Grp$ of groups and $\Ab$ of abelian groups are protomodular, while $\OrdGrp$ and $\OrdAb$ are not (see Theorem 4.6 in~\cite{CM-FM}). In this section we introduce an enriched preorder structure on morphisms which does work in favour of lax protomodularity for $\OrdAb$.

We start by analysing possible $\Ord$-enrichments for the category $\OrdGrp$ of preordered groups and monotone homomorphisms. We recall that a preordered group is a (not necessarily abelian) group $(X,+,0)$ equipped with a preorder $\le$ such that the group operation is monotone
\[x\le y, u\le v\;\; \Rightarrow\;\; x+u \le y+v,\]
for any elements $x,y,u,v\in X$; their morphisms are the monotone group homomorphisms.
The preorder of a group $(X,+,0)$ is completely determined by its \emph{positive cone}, which is the submonoid of $X$, closed under conjugation, given by its positive elements, $P_X=\{x\in X: 0\leqslant x\}$.

In $\OrdGrp$ the pointwise preorder on morphisms trivializes; that is, if one defines, for morphisms $f,g\colon X\to Y$, $f\leq g$ if, for all $x\in X$, $f(x)\leq g(x)$, then also $f(-x)\leq g(-x)$, and consequently, $\leq$ is symmetric. So, instead we use the pointwise order restricted to positive elements, and define, for morphisms $f,g\colon X \to Y$ of $\OrdGrp$,
\begin{equation}
\label{< for OrdAb}
	f\preccurlyeq g \;\;\Leftrightarrow\;\; \forall x\in P_X, f(x)\leq g(x).
\end{equation}
It is straightforward to check that (pre)composition preserves the preorder of $\OrdGrp(X,Y)$, for any preordered groups $X$ and $Y$, and so this defines an $\Ord$-enriched category $\OOrdGrp$.\\

$\OOrdGrp$ does not have precomma objects in general. In order to prove this assertion first note that, if \eqref{comma f/g} is a comma object in $\OOrdGrp$, then $C$ is isomorphic to $X\times Z$, as a group, and $c_1, c_2$ are the product projections. This follows easily from the following inequality
\[\vcenter{\xymatrix@=30pt{(X\times Z,{0}) \ar[r]^-{\pi_Z} \ar[d]_-{\pi_X} \ar@{}[dr]|-{\preccurlyeq} & Z \ar[d]^-g \\
								 X \ar[r]_-f & Y}}\]
and the universal properties of products and comma objects. So what remains to be studied is the existence of a positive cone for $X\times Z$ that makes \eqref{comma f/g} a comma object. Let $Y$ be a preordered group, $y\in P_Y$ and $\varphi\colon \Z\to Y$ with $\varphi(1)=y$, where $\Z$ is the usual ordered group of integers, and assume that the comma object of $(\varphi, 1_Y)$ exists in $\OOrdGrp$:
\[\vcenter{\xymatrix@=30pt{(\Z\times Y,P) \ar[r]^-{\pi_2} \ar[d]_-{\pi_1} \ar@{}[dr]|-{\preccurlyeq} & Y \ar[d]^-{1_Y} \\
								 \Z \ar[r]_-\varphi & Y.}}\]
Then it is easy to check that $(1,y)\in P$, and so, by conjugation, $(1,a+y-a)$ belongs also to $P$, for any $a\in Y$. Therefore $y=(\varphi\pi_1)(1,a+y-a)\leq\pi_2(1,a+y-a)=a+y-a$. Since this inequality is not valid in general, we conclude that the comma object may not exist. As an example of this failure, consider the group $Y$ of monotone bijective (and therefore continuous) endomaps of the real line $y\colon\R\to\R$ with the operation given by composition, ordered by $y\leq y'$ if, for every $x\in\R$, $y(x)\leq y'(x)$. Then, for instance, for $y,a\colon\R\to\R$ defined by $y(x)=x+1$ and $a(x)=x^3$, $y$ is positive, since, for every $x$, $x\le y(x)$, but $y\not\leq a\circ y \circ a^{-1}$.\\

To overcome this absence of precomma objects, we focus on its full subcategory $\OOrdAb$ of preordered abelian groups.
Here the preorder of an abelian group $(X,+,0)$ is completely determined by a submonoid of $X$ of its positive elements, $P_X=\{x\in X: 0\leqslant x\}$, since closedness under conjugation comes for free.
We consider in $\OOrdAb$ the $\Ord$-enrichment inherited from $\OOrdGrp$.

Then $\OOrdAb$ does not admit comma objects in general, but it admits precomma objects. As for $\OOrdGrp$, if \eqref{comma f/g} is a (pre)comma object in $\OOrdAb$, then $C$ is isomorphic to $X\times Z$, as a group, and $c_1,c_2$ are the product projections as in the diagram below:
\begin{equation}\label{diag:OrdAb}
	\xymatrix{f/g=(X\times Z,P_{f/g}) \ar[d]_-{\pi_1} \ar[r]^-{\pi_2} \ar@{}[dr]|-{\preccurlyeq} & Z \ar[d]^-g \\
						X \ar[r]_-f & Y.}
\end{equation}
The positive cone of $f/g=X\times Z$ must be \[P_{f/g}=\{(x,z)\in P_{X\times Z}:f(x)\le g(z)\}.\] Indeed, if $\alpha\colon A\to X$, $\beta\colon A\to Z$ are such that $f\alpha\preccurlyeq g\beta$, then $\langle\alpha,\beta\rangle\colon A\to X\times Z$ is monotone: for every $a\in A$, if $a\geq 0$ then both $\alpha(a)\geq 0$ and $\beta(a)\geq 0$, and, moreover, $f\alpha(a)\leq g\beta(a)$, hence $\langle\alpha,\beta\rangle(a)\in P_{f/g}$.

To check that \eqref{diag:OrdAb} is not always a comma object, consider $f=g=1_{\Z}$, where again $\Z$ is the usual ordered group of integers, so that $P_{f/g}=\{(n,m)\in\Z\times\Z\;;\;n\geq 0, m\geq 0, n\leq m\}$. The homomorphisms $t,t'\colon \Z\to\Z\times\Z$, defined by $t(n)=(n,4n)$ and $t'(n)=(3n,5n)$ for every $n\in\Z$, are monotone and such that $\pi_1 t\preccurlyeq\pi_1 t'$ and $\pi_2 t\preccurlyeq \pi_2 t'$, but $t\not\preccurlyeq t'$: $t(1)\leq t'(1)$ would imply $(2,1)=t'(1)-t(1)\in P_{f/g}$, which is false.

Note that, in the precomma object above, $\pi_2$ is split by the morphism $\langle 0,1_Z\rangle:Z\to X\times Z$ (as usual), while $\pi_1$ is not split by the group homomorphism $\langle 1_X,0\rangle\colon X\ito X\times Z$, since it is not monotone (denoted $\ito$ to emphasize it is not a morphism in $\OOrdAb$). Actually, $(\pi_2, \la 0,1_Z\ra)$ is a rali point: for any $(x,z)\in P_{f/g}$, we get $(0,z)\leq (x,z)$, because $x\in P_X$; thus $\la 0,1_Z\ra \pi_2\preccurlyeq 1_{f/g}$.\\

The following result proves the conservativeness of the vertical precomma object change-of-base functor with respect to the slice category, which implies the same property with respect to points. Thus \emph{$\OOrdAb$ is a lax preprotomodular category}.

\begin{proposition}\label{V conservative in OrdAb} For any morphism $\alpha\colon A\to Y$ in $\OOrdAb$, the precomma object functor \linebreak $V_{\alpha}\colon \OOrdAb/Y\to \OOrdAb/A$ is conservative.
\end{proposition}
\begin{proof} We build the diagram
$$
\xymatrix@C=30pt@R=40pt{A\times Z \ar[ddr]_-{\rho_1} \ar@<2pt>[rr]^-{\rho_2} \ar@{->}[dr]^(.6){V_{\alpha}(\gamma)=1_A\times \gamma}|-{\cong} & &
								Z \ar[ddr]_(.3){g} \ar[dr]^-{\gamma} \ar@<2pt>[ll]^-{\la 0,1_Z\ra}\\
								& A\times X \ar[d]_(.4){\pi_1} \ar@<2pt>[rr]^(.4){\pi_2} & & X \ar[d]^-{f} \ar@<2pt>[ll]^-{\la 0,1_X\ra}\\
								& A \ar[rr]_-{\alpha} & & Y}
$$
\begin{center}\tiny
		\hspace{75pt}($\alpha \pi_1\preccurlyeq f \pi_2$, $\;\alpha \rho_1\preccurlyeq g \rho_2$, 	$\;f\gamma = g$)
\end{center}
where the front and back faces are precomma diagrams. Note that the upper trapezoid also commutes with the precomma projection splitings, i.e. $\langle 0,1_X\rangle \gamma=V_{\alpha}(\gamma)\langle 0,1_Z\rangle$. If $V_{\alpha}(\gamma)$\linebreak $(=1_A\times \gamma)$ is an isomorphism, then $\gamma$ is also an isomorphism. The inverse of $\gamma$ is given by the composite of morphisms $\rho_2 V_{\alpha}(\gamma)^{-1}\langle 0,1_X\rangle\colon$ $X\to Z$, since
$$
\begin{array}{l}
	\gamma \rho_2 V_{\alpha}(\gamma)^{-1}\langle 0,1_X\rangle = \pi_2 V_{\alpha}(\gamma) V_{\alpha}(\gamma)^{-1} \langle 0,1_X\rangle = \pi_2 \langle 0,1_X\rangle = 1_X \\
	\rho_2 V_{\alpha}(\gamma)^{-1}\langle 0,1_X\rangle \gamma = \rho_2 V_{\alpha}(\gamma)^{-1} V_{\alpha}(\gamma) \langle 0,1_Z\rangle = \rho_2 \langle 0,1_Z\rangle = 1_Z.
\end{array}
$$
\end{proof}

The corresponding horizontal result does not hold for $\OOrdAb$, as the next example shows.

\begin{example}\label{OrdAb is no Ord-coproto} In the following diagram
\[\xymatrix@R=30pt@C=40pt{(\Z\times\Z,0) \ar@{}[dr]|-{\mbox{\caixa{1}}} \ar@<2pt>[r] \ar[d]_-{1_{\Z\times \Z}} & 0 \ar@<2pt>[l] \ar[d]\\
													(\Z\times\Z,0\times\N) \ar@{}[dr]|-{\mbox{\caixa{2}}}\ar@<2pt>[r]^-+ \ar[d]_-{1_{\Z\times \Z}} & (\Z,\N) \ar[d]^-{1_{\Z}} \ar@<2pt>[l]^-{\langle 0,1\rangle}  \\
													(\Z\times\Z,\N\times\N) \ar@<2pt>[r]^-+  &  (\Z,\N), \ar@<2pt>[l]^-{\langle 0,1\rangle}}
\]
it is easily checked that both \caixa{1} and the outer rectangle are precomma objects. Indeed the positive cone of $\Z\times\Z$ for both precomma objects is given by \[\{(z,z')\in\Z\times\Z\,;\,z\geq 0,\, z'\geq 0,\, z+z'\leq 0\}=\{(0,0)\}.\] However, \caixa{2} is commutative but it is not a pullback, showing that $\OOrdAb$ does not satisfy condition (ii) of Corollary \ref{H conservative iff |1|2|}, that is, \emph{$\OOrdAb$ is not colax preprotomodular}.
\end{example}

With respect to the horizontal precomma object change-of-base functor and its conservativeness for $\OOrdAb$, we can only prove that each $H_{\alpha}$ is conservative when applied to rali points. As mentioned above, the top projection of a precomma object of a pair of morphisms \linebreak $(f\colon X \to Y,g\colon Z\to Y)$ always gives rise to a rali point $(\pi_2,\la 0,1_Z\ra)$. However, taking the precomma object of $(f,g)$, when $(f,s)$ is a rali point, gives rise to a point $(\pi_2,\la sg,1_Z\ra)$ which is not necessarily rali. For instance, in the precomma object of $(1_\Z,1_\Z)$, $(\pi_2,\la 1_\Z,1_\Z\ra)$ is not a rali. So, the functor below goes from rali points to ordinary points.

\begin{proposition}\label{H conservative in OrdAb} For any morphism $\alpha\colon A\to Y$ in $\OOrdAb$, the functor
$$
	H^{r}_{\alpha}\colon \Ptr_Y(\OOrdAb)\to \Pt_A(\OOrdAb)
$$
is conservative.
\end{proposition}

\begin{proof} We build the diagram
$$
\xymatrix@R=30pt@C=40pt{Z\times A \ar[dd]_-{\rho_1} \ar[drrr]^-{\rho_2} \ar@{->}[dr]|(.4){H^r_{\alpha}(\gamma)=\gamma\times 1_A}|(.6){\cong} \\
								& X\times A \ar[dd]_(.4){\pi_1} \ar[rr]^(.3){\pi_2} && A \ar[dd]^-{\alpha} \\
								Z \ar@<2pt>[drrr]^(.3){g} \ar[dr]_-{\gamma} \\
								& X \ar@<2pt>[rr]^(.3){f} && Y \ar@<2pt>[ll]^(.7)s \ar@<2pt>[lllu]^(.7)t }
$$
\begin{center}\tiny
		\hspace{90pt}($\;f\pi_1\preccurlyeq \alpha\pi_2$, $\;g\rho_1\preccurlyeq \alpha\rho_2$, $\;f\gamma = g$, $\; \gamma t=s$, $sf\preccurlyeq 1_X$)
\end{center}
where the front and back faces are precomma diagrams. If $H^r_{\alpha}(\gamma)(=\gamma\times 1_A)$ is an isomorphism, then it easily follows that $\gamma$ is a monomorphism (=injective) and an epimorphism (=surjective). We cannot proceed as in the previous proof since the vertical projections of precomma objects need not be split epimorphisms. To conclude that $\gamma$ is an isomorphism, it suffices to show that $\gamma$ is a regular epimorphism, i.e. that $\gamma(P_Z)\supseteq P_X$, which gives $\gamma(P_Z)=P_X$.

For $x\in P_X$, we get $x-sf(x)\in P_X$, because $f$ is rali, thus $(x-sf(x),0)\in P_{X\times A}$. Since $f(x-sf(x))=0\le \alpha(0)$, then $(x-sf(x),0)\in P_{f/\alpha}$. It follows that $\rho_1 H^r_{\alpha}(\gamma)^{-1}(x-sf(x),0)\in P_Z$. We also have $tf(x)\in P_Z$. So, there exists an element $\overline{z}=\rho_1 H^r_{\alpha}(\gamma)^{-1}(x-sf(x),0)+tf(x)\in P_Z$, such that $\gamma(\overline{z})=x-sf(x)+\gamma tf(x)=x-sf(x)+sf(x)=x$.
\end{proof}

Finally one remark that $\Ord$-enriched regularity does not follow from the corresponding\linebreak 1-dimensional property. Indeed, as an epireflective subcategory of the regular category $\OrdGrp$ (see \cite{CM-FM} and \cite{BCGS}), $\OrdAb$ is a regular category in the sense of~\cite{Barr}, but $\OOrdAb$ is not $\Ord$-regular (cf. \cite{BG} and \cite{KV}), neither in the case we consider as right factor $\mathcal{M}$ of the factorization the fully faithful morphisms nor the monic and fully faithful morphisms. Indeed, as we show next, in both cases there is a class of morphisms $\mathcal{E}$ such that $(\mathcal{E},\mathcal{M})$ is a \emph{non-stable} orthogonal factorisation system.

\begin{proposition} Let $\mathcal{M}$ (respectively $\mathcal{M}'$) be the class of (respectively monic and) fully faithful morphisms in $\OOrdAb$ and let \[\mathcal{E}=\{h\colon A\to B\mbox{ bijective}\,;\mbox{ for all }b\in P_B, \,b=h(a'-a),\mbox{ for }a',a\in P_A\}, \mbox{ and}\] \[\mathcal{E}'=\{h\colon A\to B\mbox{ surjective}\,;\,\mbox{ for all }b\in P_B, \,b=h(a'-a),\mbox{ for }a',a\in P_A\}.\]
\begin{enumerate}
\item A morphism $f\colon X\to Y$ in $\OOrdAb$ belongs to $\mathcal{M}$ if, and only if, for all $x,x'\in P_X$, $x\leq x'\;\Leftrightarrow\;f(x)\leq f(x')$; a fully faithful morphism $f$ belongs to $\mathcal{M}'$ if, in addition, it is an injective map.
\item Given a commutative diagram
\begin{equation}\label{diag:square}
\xymatrix{A\ar[d]_h\ar[r]^u&X\ar[d]^f\\
B\ar[r]_v&Y,}
\end{equation}
with $h\in\mathcal{E}'$ and $f\in\mathcal{M}'$, there exists a unique morphism $d\colon B\to X$ such that $dh=u$ and $fd=v$.
\item Given a commutative diagram \eqref{diag:square}, with $h\in\mathcal{E}$ and $f\in\mathcal{M}$, there exists a unique morphism $d\colon B\to X$ such that $dh=u$ and $fd=v$.
\item Every morphism in $\OOrdAb$ factors through a morphism in $\mathcal{E}'$ followed by a morphism in $\mathcal{M}'$.
\item Every morphism in $\OOrdAb$ factors through a morphism in $\mathcal{E}$ followed by a morphism in $\mathcal{M}$.
\item $(\mathcal{E},\mathcal{M})$ and $(\mathcal{E}',\mathcal{M}')$ are non-stable factorization systems.
\end{enumerate}
\end{proposition}
\begin{proof}
(1) If $f\colon X\to Y$ is fully faithful and $x,x'\in P_X$ are such that $f(x)\leq f(x')$, then $\varphi, \varphi'\colon \Z\to X$ defined by $\varphi(1)=x$ and $\varphi(1)=x'$ are morphisms in $\OOrdAb$ such that $f\varphi\preccurlyeq f\varphi'$, and so $\varphi\preccurlyeq\varphi'$, or, equivalently, $x\leq x'$. Conversely, if for all $x,x'\in P_X$, $x\leq x'\;\Leftrightarrow\;f(x)\leq f(x')$ and $f g\preccurlyeq f h$, for $g,h\colon W\to X$, then, for every $w\in P_W$, $f(g(w))\leq f(h(w))$ and so $g(w)\leq h(w)$, that is $g\preccurlyeq h$.\\

(2) Given \eqref{diag:square} with $h\in\mathcal{E}'$ and $f\in\mathcal{M}'$, since $h$ is surjective and $f$ is injective, we define the homomorphism $d\colon B\to X$ as usual in $\Ab$: $d(b)=u(a)$ for any $a\in h^{-1}(b)$. Then $d$ is the unique map such that $d h=u$ and $f d=v$. It remains to show that it is monotone: if $b\in P_B$, then $b=h(a'-a)$ for some $a,a'\in P_A$, and $d(b)=u(a'-a)$. From $v(b)=f(u(a'-a))\in P_Y$, we get $u(a'-a)\in P_X$, since $f$ is fully faithful; thus $d(b)=u(a'-a)\in P_X$.\\

(3) An analogous argument shows that morphisms in $\mathcal{E}$ are orthogonal to fully faithful morphisms: given \eqref{diag:square} with $h\in\mathcal{E}$ and $f\in\mathcal{M}$, since $h$ is bijective we define $d\colon B\to X$ by $d(b)=u(a)$, where $a$ is the unique element of $h^{-1}(b)$. Monotonicity of $d$ follows from arguments similar to those used above.\\

(4) Every morphism $g\colon Z\to Y$ factors as $\xymatrix{Z\ar[r]^-{g'}&(g(Z),P')\ar[r]^-m&Y}$, where\linebreak $P'=\{y\in P_Y\,;\;y=g(z'-z)$ for $z,z'\in P_Z\}$, with, by construction, the corestriction $g'$ of $g$ in $\mathcal{E}'$ and the inclusion $m$ in $\mathcal{M}'$.\\

(5) Analogously, every morphism $g\colon Z\to Y$ factors as $\xymatrix{Z\ar[r]^-{1_Z}&(Z,P)\ar[r]^-{\tilde{g}}&Y}$, where\linebreak $P=\{z\in Z; g(z)\in P_Y$ and $z=z''-z'$ for $z',z''\in P_Z\}$, with the identity $1_Z$ in $\mathcal{E}$, and $\tilde{g}$ defined as $g$, which is fully faithful due to the way $P$ is defined.\\

(6) Since all these classes are closed under composition with isomorphisms, we may conclude that both pairs are factorization systems in $\OOrdAb$ (cf. \cite[Definition 14.1]{AHS}).

To show that they are not stable, we consider the following pullback
\[\xymatrix{(\Z,\{0\})\ar[r]^1\ar[d]_1&(\Z,-\N)\ar[d]^1\\
(\Z,\N)\ar[r]_1&(\Z,\Z),}\]
where $(\Z,\N)\to (\Z,\Z)$ belongs to both $\mathcal{E}$ and $\mathcal{E}'$ but $(\Z,\{0\})\to(\Z,-\N)$ does not belong to either of them.
\end{proof}

\begin{remark}
In \cite{CM21} the category $\VGrp$ of $V$-groups and $V$-homomorphisms, for a commutative and unital quantale $V$, was studied. (The reader may want to give a look at \cite{CM21}, and the subsequent paper \cite{Mi}, to know more on $V$-groups.)

Very briefly, we point out that what we have done for $\OrdAb$ can be generalized for the category $\VAb$ of abelian $V$-groups and $V$-homomorphisms, for a commutative and unital quantale $V$. The $\Ord$-enrichment in $\VAb$ is defined, for $V$-homomorphisms $f,g\colon (X,a)\to(Y,b)$ by $f\preccurlyeq g$ if, for all $x\in X$, $a(0,x)\leq b(f(x),g(x))$ in $V$; as for $\OrdAb$, we denote this $\Ord$-enriched category by $\OVAb$. This category has precomma objects: given morphisms $f\colon (X,a)\to(Y,b)$ and $g\colon (Z,c)\to(Y,b)$ in $\OVAb$, the precomma object is defined as in the following diagram
\[
	\xymatrix{(X\times Z,d) \ar[d]_-{\pi_1} \ar[r]^-{\pi_2} \ar@{}[dr]|-{\preccurlyeq} & (Z,c) \ar[d]^-g \\
						(X,a) \ar[r]_-f & (Y,b),}
\]
where $d((x,z),(x',z'))=a(x,x')\wedge c(z,z')\wedge b(f(x'-x),g(z'-z))$, for every $(x,z), (x',z')\in X\times Z$. Since $d$ makes $X\times Z$ a $V$-category and it is invariant under shifting, $(X\times Z,d)$ is a $V$-group by \cite[Proposition 3.1]{CM21}. Moreover, from the definition of $d$ it follows that the projections $\pi_1$ and $\pi_2$ are $V$-homomorphisms, and that $f\pi_1\preccurlyeq g\pi_2$. The universal property of this diagram is easily checked, that is, conditions (C1) and (C2) are satisfied.

Now it is clear that the inclusion
\[\xymatrix{\OrdAb\ar[r]&\VAb}\]
becomes $\Ord$-enriched and preserves precomma objects, and so we may conclude directly that $\OVAb$ does not have comma objects in general.

From the failure of lax preprotomodularity for $\OOrdAb^\co$ it follows immediately that $\OVAb^\co$ is not lax preprotomodular. Still, using exactly the arguments of the proof of Proposition \ref{V conservative in OrdAb} and the description of precomma objects above, it is straightforward to conclude that $\OVAb$ is lax preprotomodular.

\end{remark}

\section*{Acknowledgments}
The authors are grateful to John Bourke for his interesting tips on 2-dimen\-sional regularity, Graham Manuell for useful discussions on $\Ord$-enrichments on toposes, and mostly to Peter T. Johnstone for his thorough thoughts on topos $\Ord$-enrichments, which gave rise to the notes \cite{PTJ}.

\end{document}